\documentclass[12pt]{amsart}
\usepackage{amsmath,amsthm,amscd,amsfonts,amssymb,epic,eepic,bbm,tikz-cd,mathtools}
\usepackage{mathrsfs}
\usepackage{hyperref}
\usepackage{ulem}
\usepackage{xcolor}

\usepackage{mathtools}
\usepackage[noabbrev]{cleveref}
\usepackage{geometry}
\newtheorem{thm}{Theorem}[section]
\newtheorem{prop}[thm]{Proposition}
\newtheorem{lem}[thm]{Lemma}

\theoremstyle{definition}

\theoremstyle{remark}
\newtheorem{remark}[thm]{Remark}

\numberwithin{equation}{section}

\usepackage{enumerate}
\usepackage{verbatim}
\usepackage{enumitem}

\usepackage[alphabetic]{amsrefs}

\newcommand{\Q}{\mathbb{Q}}

\newcommand{\bP}{\mathbf{P}}
\newcommand{\G}{\mathfrak{G}}
\newcommand{\E}{\mathcal{E}}

\numberwithin{equation}{subsection}

\allowdisplaybreaks
\allowdisplaybreaks

\begin{document}

\title[Euler Characteristics of the Generalized Kloosterman sheaves]{Euler Characteristics of the Generalized Kloosterman sheaves for symplectic and orthogonal groups}
\author{Yu Fu}
\address{Department of Mathematics\\
California institute of technology\\
Pasadena, CA, 91125}
\email{yufu@caltech.edu}

\author{Miao (Pam) Gu}
\address{Department of Mathematics\\
University of Michigan\\
Ann Arbor, MI 48109-1043}
\email{pmgu@umich.edu}


\begin{abstract}
We study the monodromy of certain $\ell$-adic local systems attached to the generalized Kloosterman sheaves constructed by Yun and calculate their Euler characteristics under standard representations in the cases of symplectic and split/quasi-split orthogonal groups. This provides evidence for the conjectural description of their Swan conductors at $\infty$ which is predicted by Reeder-Yu on the Langlands parameters attached to the epipelagic representations. 
\end{abstract}

\maketitle

\section{Introduction}

Let $k=\mathbb{F}_q$ be a finite field of characteristic $p$ and $F=k(t)$ be the rational function field over $k$. Places of F are in natural bijections with closed points of $X:= \mathbb{P}_{k}^1$. Let $G$ be a reductive quasi-split group over $F$. Let $K=F_{\infty}$ be the local field of $F$ at $\infty$. In \cite{ReederYu}, Reeder and Yu constructed a family of supercuspidal representations of $p$-adic groups called
\textit{epipelagic representations}, and later Yun \cite{YunEpipe} constructed automorphic representations for quasi-split groups over $F$, whose local components at $\infty$ are epipelagic representation. They also constructed the associated Galois representations from an admissible parahoric subgroup $\mathbf{P} \subset G(K)$ under the geometric Langlands correspondence. The Galois representations give ${}^{L}G$-local systems which generalizes the Kloosterman sheaves constructed earlier by Heinloth, Ngo and Yun \cite{HNY}. 

In this paper, we compute Euler characteristics of the $\ell$-adic local systems attached to the generalized Kloosterman sheaf and standard representation of $\widehat{G}$ in the cases that $G$ is some symplectic group or orthogonal group. This is equivalent to calculating the Swan conductor at $\infty$ and gives evidence for the prediction made by Reeder and Yu about the Langlands parameters of epipelagic representations. 

\subsection{The Kloosterman sheaves} \label{kloointro}

In \cite{HNY}, generalizing Deligne's and Katz's constructions of the Kloosterman sheaves attached to families of Kloosterman sums, Heinloth, Ng$\widehat{\mathrm{o}}$, and Yun constructed Kloosterman sheaves $\mathrm{Kl}_{\widehat{G}, \mathbf{I}}(\mathcal{K},\phi)$ for reductive groups $G$ and $\mathbf{I}$ the Iwahoric subgroup of $G$, which are $\widehat{G}$-local systems. The Kloosterman sheaves depend on $\mathcal{K}$ a character sheaf and $\phi$ a stable linear functional, which corresponds to simple supercuspidal representations of Gross-Reeder \cite{GrossReeder} on the automorphic side. 

Generalizing the construction in \cite{HNY}, Yun constructed in \cite{YunEpipe} a family of generalized Kloosterman sheaves  $\mathrm{Kl}_{\widehat{G},\mathbf{P}}(\mathcal{K},\phi)$ associated with parahoric subgroups $\mathbf{P}$. We recall some details of the construction below. 

We fix a prime $\ell$ different from $\mathrm{char}(k)$. We consider constructible $\overline{\mathbb{Q}}_{\ell}$-complexes over various algebraic stacks over $k$ or $\overline{k}$. Let $(L_{\mathbf{P}}, V_{\mathbf{P}})$ be $\theta$-groups studied by \cite{Vinberg}. When $\mathrm{char}(k)$ is large, let $\mathbf{P}:= \mathbf{P}_m$, where $m$ is a regular elliptic number uniquely associated with the admissible parahoric $\mathbf{P}$ (see \Cref{regular elliptic number}). 

Let $\pi(\chi,\phi)$ be a cuspidal automorphic representation whose local component at $\infty$ is an epipelegic supercuspidal representation as realized in \cite{YunEpipe}, where $\chi$ is a character of $\widetilde{L_\mathbf{P}}$ and $\phi$ is a stable linear functional $V_\mathbf{P} \to k$.

By the global geometric Langlands correspondence, it is associated with the Kloosterman sheaf $\mathrm{Kl}_{{}^{L}G,\mathbf{P}}(\mathcal{K},\phi)$ (which is a ${}^{L}G$-local system), where $\mathcal{K}$ is the character sheaf associated with $\chi$ under sheaf-to-function dictionary. 
Let $X^{\circ}=\mathbb{P}_k^1-\{0,\infty\}$. The local system attached to the generalized Kloosterman sheaves and stardard representation $V$ of $\widehat{G}$ can be interpreted using the following diagram 
\begin{equation*}
\begin{tikzcd}
 & X^{\circ} \times \mathfrak{G}_{\leq \lambda} \arrow{dl}[swap]{f_{\phi}} \arrow[dr, "\pi"] \\
\mathbb{A}^1  && X^{\circ}
\end{tikzcd}.
\end{equation*}
Here $\mathfrak{G}_{\leq \lambda}$ is the variety corresponding to the dominant coweight $\lambda$ associated with $V$ and $\pi$ is the projection map. When $G=\mathrm{Sp}_{2n}$, $\mathfrak{G}_{\leq \lambda} = \mathrm{Sym}^2_{\leq 1}(M)$, where $M$ is $2n$-dimensinal symplectic space. The $\overline{\mathbb{Q}}_{\ell}$-local system $\mathrm{Kl}^{\mathrm{St}}_{\widehat{G},\mathbf{P}}(\mathbf{1},\phi)$ asssoicated to the Kloosterman sheaves $\mathrm{Kl}_{\widehat{G},\mathbf{P}}(\mathbf{1},\phi)$ (the $\widehat{G}$-local system descended from $\mathrm{Kl}_{^{L}G,\mathbf{P}}(\mathcal{K},\phi)$) with the trivial character sheaf $\mathcal{K} = \mathbf{1}$ and the standard representation of $G$ are 
\[
\mathrm{Kl}^{\mathrm{St}}_{\widehat{G},\mathbf{P}}(\mathbf{1},\phi) \cong \pi ! f_{\phi}^{*} \mathrm{AS}_{\psi}[2n-1]\left(\frac{2 n-1}{2}\right).
\]
When $G$ is an orthogonal group, let $(M,q)$ be the associated quadratic spaces of dimension $2n$ or $2n+1$. The variety $\mathfrak{G}_{\leq \lambda} \subset \mathbb{P}(M)$ is defined by $Q(q)-\cup Q_{[i,m-i]}$, where $Q(q)$ is the quadric defined by $q=0$ and $Q_{[i,m-i]}$ is the quadric defined by $q=0$ when restricted to $M_{i}\oplus\cdots\oplus M_{m-i}$. We have
\[
\mathrm{Kl}^{\mathrm{St}}_{\widehat{G},\mathbf{P}}(\mathbf{1},\phi) \cong \pi ! f_{\phi}^{*} \mathrm{AS}_{\psi}[\mathrm{dim} M-2]\left(\frac{\mathrm{dim} M -2 }{2}\right).
\]

\subsection{Main Theorem}
In this paper we prove the following main theorems:
\begin{thm}\label{Echar symplectic}
	When $G=\mathrm{Sp}_{2n}$, we have 
	$$ - \chi_{c}(\widetilde{X}^{\circ}, \mathrm{Kl}^{\operatorname{st}}_{\widehat{G}, \mathbf{P}_{m}}(\chi , \phi))=  d, $$
	where $d=\frac{2n}{m}$.
\end{thm} 

\begin{thm}\label{Echar orthogonal}
	When $G$ are split or quasi-split orthogonal groups whose root systems are of either type $B_n$, $D_n$, or $^{2}D_n$, we have 
	$$
-\chi_{c}(\widetilde{X}^{\circ}, \mathrm{Kl}^{\operatorname{st}}_{\widehat{G}, \mathbf{P}_{m}}(\chi , \phi)) = \begin{cases} 2d &  B_n, \\ 2d &  D_n, \text{ } \omega_1 \text{ is non-degenerate,} \\ 2(d+1) &  D_n, \text{ } \omega_1 \text{ is degenerate,} \\ 2d & ^{2}D_n,\text{ } \omega_1 \text{ is non-degenerate,} \\ 2(d+1) & ^{2}D_n,\text{ } \omega_1 \text{ is degenerate.}\end{cases}.
$$
Here $d=\frac{2n}{m}$ when $G$ is of type $B_n$, $d | n$ is even for type $D_n$, and $d | n$ is odd for type $^2D_n$; $d=\frac{2(n-1)}{d}$ is odd when $G$ is of type $D_n$ and even for type $^2D_n$. 
\end{thm} 

Using the Grothendieck-Ogg-Shafarevich formula, $-\chi_{c}(\widetilde{X}^{\circ}, \mathrm{Kl}^{\operatorname{st}}_{\widehat{G}, \mathbf{P}_{m}}(\chi , \phi))$ is equivalent to the Swan conductor of $\mathrm{Kl}^{\operatorname{st}}_{\widehat{G}, \mathbf{P}_{m}}(\chi , \phi)$ at $\infty$. Our main theorems then provide evidence for Reeder-Yu's conjecture \cite{ReederYu} about the Swan conductor of the Langlands parameter of epipelagic representations of $G$ in the function field case. 

\begin{remark}
As mentioned in \Cref{kloointro}, when $\mathbf{P}=\mathbf{I}$, the local system matches with the local system in \cite{HNY}. In that case, the corresponding regular elliptic number $m$ is the twisted Coxeter number (see \Cref{regular elliptic number}). It was shown in \cite[Corollary 5.1]{HNY} that $\mathrm{Swan}_{\infty}(\mathrm{Kl}^{V_{\theta^{\vee}}}_{\widehat{G}, \mathbf{I}}(\mathcal{K},\phi)) = r_s(\widehat{G})$, where $V_{\theta^{\vee}}$ is the unique quasi-minuscule representation of $\widehat{G}$ whose nonzero weights consist of short roots of $\widehat{\mathfrak{g}}$ and $r_s(\widehat{G})$ is the number of short simple roots of $\widehat{G}$. When $G=\mathrm{Sp}_{2n}$, the unique quasi-minuscule representation of $\widehat{G}$ is the standard representation of $\widehat{G}$ (see \cite{Ses78} for reference). In \Cref{Echar symplectic}, when $d=1$, $m$ is the twisted Coxeter number, and \Cref{Echar symplectic} matches up with this result. 
\end{remark}

\subsection{Outline of the paper}
We set up the notations and review details about the epipelagic representations and the construction of the generalized Kloosterman sheaves in \Cref{Sec2}. We discuss the relation between the Euler characteristics of the local system and the Swan conductor of the conjectural Langlands parameter of epipelagic representations in \Cref{Sec3}. In \Cref{Sec4} and \Cref{Sec5}, we prove the main theorems in the symplectic case and orthogonal case, respectively.

\section*{Acknowledgements}
This project was suggested to us by Zhiwei Yun in Arizona Winter School 2022. We would like to thank Zhiwei Yun for his continued support during our work on the project. We would like to thank the organizers of Arizona Winter School 2022 for an excellent event, and for giving us the opportunity to work on this project. The first author would like to thank Yi Ni for helpful discussions. The second author would like to thank Robert Cass for answering many questions on derived categories, and Charlotte Chan, Tasho Kaletha, Jianqiao Xia, and Zeyu Wang for helpful discussions.

\section{Preliminaries} \label{Sec2}

In this section, we recall some details from \cite{YunEpipe} about the realization of epipelagic representations as a local component of an automorphic representations of $G(\mathbb{A}_F)$ and the construction of the Kloosterman sheaves.

\subsection{The absolute group data and the Langlands dual group}

Let $\mathbb{G}$ be a a split reductive group over $k$ whose derived group is almost simple. Fix a pinning $\dagger = (\mathbb{B}, \mathbb{T}, \dots)$ of $\mathbb{G}$ for $\mathbb{B}$ a Borel subgroup and $\mathbb{T} \subset \mathbb{B}$ a split torus. Let $\mathbb{W}=N_{\mathbb{G}}(\mathbb{T})/\mathbb{T}$ be the Weyl group of $\mathbb{G}$ and let $\sigma \in \mathrm{Aut}^{\dagger}(\mathbb{G})$ be the image of $1$. Fix a cyclic subgroup $\mathbb{Z}/e\mathbb{Z}$ of the pinned automorphism group of $\mathbb{G}$. We assume $\mathrm{char}(k)$ is prime to $e$ and $k^{\times}$ contains $e$th roots of unity $\mu_e$.  

Fix a $\mu_e$-cover $\widetilde{X} \to X$ which is totally ramified over $0$ and $\infty$. $\widetilde{X}$ is isomorphic to $\mathbb{P}_k^1$ with affine coordinate $t^{1/e}$. We denote  
\[
X^{\circ} := X - \{0,\infty\}; \widetilde{X}^{\circ}:= \widetilde{X} - \{0,\infty\}.
\]

Let $\widehat{G}$ be the reductive group over $\overline{Q}_{\ell}$ whose root system is dual to that of $\mathbb{G}$. We define the Langlands dual group ${}^{L}G$ to be ${}^{L}G = \widehat{G} \rtimes \mu_{e}$.

\subsection{Admissible parahoric subgroups and regular elliptic numbers}\label{regular elliptic number}

We denote $K=F_{\infty}$ as the local field of $F$ at $\infty$. Let $\mathbf{P} \subset G(K)$ be a standard parahoric subgroup containing the standard Iwahori subgroup $\mathbf{I}$. Let $\mathfrak{A}$ be the apartment in the building of $G(K)$ corresponding to the maximal split torus of $G(K)$. We denote $\Psi_{\mathrm{aff}}$ as the set of affine roots of $G(K)$, which are certain affine functions on $\mathfrak{A}$.  $\mathbf{P}$ determines a facet $\mathfrak{F}_\mathbf{P}$ in $\mathfrak{A}$, and let $\xi \in \mathfrak{A}$ denotes its barycenter. Let $m=m(\mathbf{P})$ be the smallest positive integer such that $\alpha(\xi) \in \frac{1}{m}\mathbb{Z}$ for all affine roots $\alpha \in \Psi_{\mathrm{aff}}$ (see \cite[\S 3.3]{ReederYu}). 

Let $\mathbf{P}  \supset \mathbf{P}^{+} \supset \mathbf{P}^{++}$ be the first three steps in the Moy-Prasad filtration of $\mathbf{P}$. Here $\mathbf{P}^{+}$ is the pro-unipotent radical of $\mathbf{P}$ and $L_{\mathbf{P}}:=\mathbf{P}/\mathbf{P}^{+}$ is the Levi factor of $\mathbf{P}$. We denote $V_{\mathbf{P}}=\mathbf{P}^{+}/\mathbf{P}^{++}$.

Let $\mathbb{W}' = \mathbb{W} \rtimes \mu_e$. An element $w \in \mathbb{W}'$ is $\mathbb{Z}$-\textit{regular} if it permutes the roots freely and \textit{elliptic} if $\mathbb{X}^{\ast}(\mathbb{T}^{\mathrm{ad}})^{w}=0$. The order of a $\mathbb{Z}$-regular elliptic element in $\mathbb{W}\sigma \subset \mathbb{W}'$ is called a \textit{regular elliptic number} of the pair $(\mathbb{W}',\sigma)$. 

$\mathbf{P}$ is called admissible if there exists a closed orbit of $L_{\mathbf{P}}$ on the dual space $V_{\mathbf{P}}^{\ast}$ with finite stabilizers. When $\mathbf{P}$ is admissible and $\mathrm{char}(k)$ is large, it was shown in \cite[\S 2.6]{YunEpipe} that there is a bijection between admissible parahoric subgroups $\mathbf{P}$ and regular elliptic numbers $m$ using \cite[Proposition 1]{GLRY},\cite[Proposition 6.4(iv)]{Spr74} and \cite[Corollary 5.1]{ReederYu}. 

\begin{remark}
   In particular, in the case of \cite{HNY}, the regular elliptic number corresponding to $\mathbf{I}$ is the twisted Coxeter number $h_{\sigma}$. 
\end{remark}

\subsection{Epipelagic representations and automorphic representations} \label{Epipe and auto rep}

Let $\psi: k \to \mathbb{Q}_{\ell}(\nu_p)^{\times}$ be a nontrivial character and $\phi:V_{\mathbf{P}} \to k$ be a stable functional. As shown in \cite[Proposition 2.4]{ReederYu}, the simple summands of the compact induction 
\[c-\mathrm{Ind}_{\mathbf{P^{+}}}^{G(K)}(\psi\circ \phi)\]
are called epipelagic representations of $G(k)$ attached to the parahoric $\mathbf{P}$ and the stable functional $\phi$. 

Let $\mathbf{I}_0 \subset G(F_0)$ be the Iwahori subgroup corresponding to the opposite Borel $\mathbf{B}^{\mathrm{opp}}$ of $G$, and let $\mathbf{P}_0 \subset G(F_0)$ be the parahoric subgroup containing $\mathbf{I}_0$ of the same type as $\mathbf{P}_{\infty}$. Let $\widetilde{\mathbf{P}}_0$ be the normalizer of $\mathbf{P}_0$ in $G(F_0)$ and $\widetilde{L}_\mathbf{P} = \widetilde{\mathbf{P}}_0/ \mathbf{P}_0^{+}$. Let $L_{\mathbf{P}}^{\mathrm{sc}}$ be the simply connected cover of the derived group of $L_{\mathbf{P}}$, and let $L_{\mathbf{P}}(k)'=\mathrm{Im}(L_{\mathbf{P}}^{\mathrm{sc}}(k) \to L_{\mathbf{P}}(k))$.  We fix a character $\chi: \widetilde{L}_{\mathbf{P}}(k)/L_{\mathbf{P}}(k) \to \overline{\mathbb{Q}}_{\ell}^{\times}$.

Let $\pi = \otimes'_{x \in |X|} \pi_x$ of $G(\mathbb{A})$ be an automorphic representation whose local components satisfying the following conditions: $\pi_{\infty}$ is an epipelagic representation attached to $\mathbf{P}_{\infty}^{+}$ and $\phi$; $\pi_0$ has an eigenvector under $\widetilde{\mathbf{P}}_0$ on which it acts through $\chi$ via the quotient $\tilde{L}_{\mathbf{P}}(k)$; $\pi_x$ are unramified for $x\neq 0, \infty$. 

It was shown in \cite[Proposition 2.11]{YunEpipe} that there is a unique cuspidal automorphic representation $\pi = \pi(\chi,\phi)$ of $G(\mathbb{A}_F)$ satisfying the above conditions. 

 \subsection{The generalized Kloosterman sheaves.}
Let $\operatorname{Bun}=\operatorname{Bun}(\mathbf{P}_{0}^{opp},\mathbf{P}_{\infty}^{++})$ be the moduli stack of bundles with $\mathbf{P}^{opp}$-level structure at $0$ and $\mathbf{P}^{++}$-level structure at $\infty$. The Hecke correspondence for Bun classifies $(x, \mathcal{E}, \mathcal{E}^{\prime}, \tau)$ where $x \in \widetilde{X}^{\circ}$, $\mathcal{E}$, $\mathcal{E}^{\prime} \in \operatorname{Bun}$ and $\tau: \mathcal{E}|_{\widetilde{X}-x} \to \mathcal{E}^{\prime}|_{\widetilde{X}-x}$ is an isomorphism of $G$-torsors preserving the level structures at $0$ and $\infty$. The moduli stack $\operatorname{Bun}_{G}(\widetilde{\mathbf{P}}_{0}, \mathbf{P}^{+}_{\infty})$ has a unique relevant point $\mathcal{E}$ with trivial automorphism group. Let $j: V_{\mathbf{P}} \hookrightarrow \operatorname{Bun}$ be the open embedding defined by $\mathcal{E}$ and AS$_{\psi}$ the pullback along the canonical pairing $<-,->:V_{\mathbf{P}} \times V_{\mathbf{P}}^{*,st} \to \mathbb{G}_{a}$ of the Artin-Schreier sheaf on $\mathbb{G}_{a}$ corresponding to a fixed character $\psi: k \to \bar{\Q}_{\ell}$. According to \cite[Section 3]{YunEpipe}, the Hecke eigensheaf can be described as 
 $$\mathcal{A}=(j \times id_{V_{\mathbf{P}}^{*,st}})_{!}\operatorname{AS}_{\psi}.$$
 
 \noindent \textbf{The generalized Kloosterman sheaves.} Let $\mathfrak{G}$ be the group of automorphisms of $\E|_{X-{1}}$ preserving the level structure at $0$ and $\infty$. Let $\lambda$ be a dominant coroot of $G$ and $V_{\lambda}$ be the irreducible representation of $\widehat{G}$ of highest weight $\lambda$. We have an affine Schubert variety $\operatorname{Gr}_{\le \lambda}$ in the affine Grassmannian $\operatorname{Gr}=L\mathbb{G}/L^{+}\mathbb{G}$ defined by $\lambda$. Let $\G_{\le \lambda}$ be the preimage of $L^{+}\mathbb{G}\backslash\operatorname{Gr}_{\le \lambda} \subset L^{+}\mathbb{G}\backslash L\mathbb{G}\slash L^{+}\mathbb{G}$ under the evaluation map $\operatorname{ev}_{\widetilde{1}}$ at the preimage of $1 \in X$
 $$\operatorname{ev}_{\widetilde{1}}: \G \to L^{+}_{\widetilde{1}}\mathbb{G}\backslash L_{\widetilde{1}}\mathbb{G}\slash L^{+}_{\widetilde{1}}\mathbb{G}.$$
 The geometric Satake equivalence gives the intersection complex $IC_{\lambda}$ which corresponds to $V_{\lambda}$. Using the evaluation maps at $0$ and $\infty$ and compose with the projections $\widetilde{\bP}_{0} \to \widetilde{L}_{\bP}^{ab}$ and $\bP^{+}_{\infty} \to V_{\bP}$ we get 
 $$(f^{\prime},f^{\prime\prime}):\G_{\le \lambda} \to \widetilde{L}_{\bP}^{ab} \times V_{\bP}.$$
 By the existence of Hecke eigensheaves over the stable functionals $V_{\bP}^{*, st}$ with eigenvalue $\operatorname{Kl}_{\widehat{G}, \bP}^{V}(\phi)$, see \cite[Lemma 3.6, Corollary 3.7]{YunEpipe}, applying the Fourier-Deligne transform we have \cite[Proposition 3.9]{YunEpipe}:
  Let $\operatorname{Four}_{\psi}: D_{c}^{b}\left(V_{\mathbf{P}}\right) \rightarrow D_{c}^{b}\left(V_{\mathbf{P}}^{*}\right)$ be the Fourier-Deligne transform (without cohomological shift). We have
$$
\left.\mathrm{Kl}_{\widehat{G}, \mathbf{P}}^{V_{\lambda}}(\chi) \cong \operatorname{Four}_{\psi}\left(f_{!}^{\prime \prime}\left(f^{\prime *} \mathcal{L}_{\chi} \otimes \operatorname{ev}_{\widetilde{1}}^{*} \mathrm{IC}_{\lambda}\right)\right)\right|_{V_{\mathbf{P}}^{*, st}}.
$$

\section{Euler characteristic of the local system and Swan conductor of epipelagic parameters} \label{Sec3}

Let $K=F_{\infty}$ be the local field of $F$ at $\infty$. In this section, we relate the computation of the Euler characteristic of the local system corresponds to the generalized Kloosterman sheaves to Reeder-Yu's prediction on the Swan conductor of the Langlands parameter of epipelegic representations.

\subsection{Epipelegic representations and Kloosterman sheaves} \label{Sec3.1}

\textcolor{red}{}

Let $\pi_{\infty}$ be the epipelagic supercuspidal representation of $G(k)$ as the local component at $\infty$ of an automorphic representation $\pi(\chi,\phi)$ as realized in \cref{Epipe and auto rep}. As shown in \cite[Corollary 3.10]{YunEpipe}, the ${}^{L}G$-local system $\mathrm{Kl}_{{}^{L}G,\mathbf{P}}(\mathcal{K},\phi)$ is the global Langlands parameter attached to $\pi(\chi,\phi)$. Here $\mathcal{K}$ is the character sheaf corresponding to $\chi$ under sheaf-to-function correspondence.

\subsection{Expected Langlands parameter}

Let $\pi_{\infty}$ be as defined in the previous subsection. The local Langlands conjecture suggests that there should be a Galois representation $\rho_{\pi_{\infty}}: W(K^s/K) \to {}^{L}G(\overline{\mathbb{Q}}_{\ell})$ attached to $\pi_{\infty}$ as the conjectural Langlands parameter. Based on the conjecture relating adjoint gamma factors and formal degrees in \cite{GrossReeder}, Reeder and Yu predicted in \cite[Section 7.1]{ReederYu} that 
\[
\mathrm{Swan}(\widehat{\mathfrak{g}}) = \frac{\# \Phi}{m}. 
\]

\subsection{Swan conductor and the Euler characteristic} \label{Sec3.2}
We have the following lemma of the Swan conductor of the local system:
\begin{lem}
For $\tilde{X}^{\circ} \cong \mathbb{P}_k^1-\{0, \infty\}$, 
\[
\chi_{c}(\tilde{X}^{\circ}, \mathrm{Kl}^{V_{\lambda}}_{\widehat{G}, \mathbf{P}_{m}}(\chi , \phi)) = -\mathrm{Swan}_{\infty}(\mathrm{Kl}^{V_{\lambda}}_{\widehat{G}, \mathbf{P}_{m}})
\]
\end{lem}

\begin{proof}
By the Grothendieck-Ogg-Shafarevich formula, we have 
\begin{align*}
    \chi_{c}(\tilde{X}^{\circ}, \mathrm{Kl}^{V_{\lambda}}_{\widehat{G}, \mathbf{P}_{m}}(\chi , \phi)) = & \chi_c(\mathbb{G}_m)\mathrm{rk}( \mathrm{Kl}^{V_{\lambda}}_{\widehat{G}, \mathbf{P}_{m}})-\sum_{x \in \{0, \infty\}} \mathrm{Swan}_x( \mathrm{Kl}^{V_{\lambda}}_{\widehat{G}, \mathbf{P}_{m}}) \\
    = & -\mathrm{Swan}_0( \mathrm{Kl}^{V_{\lambda}}_{\widehat{G}, \mathbf{P}_{m}}) - \mathrm{Swan}_{\infty}( \mathrm{Kl}^{V_{\lambda}}_{\widehat{G}, \mathbf{P}_{m}}), 
\end{align*}
Since $0$ is the tame point as shown in \cite[Theorem 4.5]{YunEpipe}, $\mathrm{Swan}_0( \mathrm{Kl}^{V_{\lambda}}_{\widehat{G}, \mathbf{P}_{m}})=0$. Then we deduce the lemma.
\end{proof}

Thus, computing the Euler characteristic of the local system corresponding to the Kloosterman sheaves is equivalent to computing the swan conductor at $\infty$. Therefore, combining this with \Cref{Sec3.1} and \Cref{Sec3.2}, the computation of $\chi_{c}(\tilde{X}^{\circ}, \mathrm{Kl}^{\mathrm{St}}_{\widehat{G}, \mathbf{P}_{m}}(\chi , \phi))$ gives evidence to Reeder-Yu's prediction.

\section{Euler Characteristic for Symplectic Groups} \label{Sec4}

The goal of this section is to prove Theorem \ref{Echar symplectic}. Throughout the section we work over $\bar{k}$ and ignore all the Tate twists. Let $G$ be a symplectic group. We first make a summary of \cite[Sec. 7]{YunEpipe} in subsection \ref{symplectic setup} and \ref{symplectic local system} where the local system attached to the Kloosterman sheaf $\operatorname{Kl}_{\widehat{G}, \mathbf{P}}(\mathcal{K}, \phi)$ and the standard representation of $\widehat{G}$ is calculated, then make an inductive argument to calculate the Euler characteristic of the local system.

 \subsection{The set-up}\label{symplectic setup}
  Let $(M, \omega)$ be a symplectic vector space of dimension $2n$ over $k$. One can extend $\omega$ linearly to a symplectic form on $M \otimes K$ and we denote  $\mathbb{G}=\operatorname{Sp}(M, \omega)$ and $G=\operatorname{Sp}(M \otimes K, \omega)$. Since regular elliptic numbers $m$ of $\mathbb{W}$ in this case are in bijection with divisors $d|n$[Yun, 4.8], we have $m=2n/d$ and let $\ell=n/d$. We fix a decomposition 
 $$M=M_{1}\oplus M_{2} \oplus \cdots \oplus M_{\ell} \oplus M_{\ell+1} \oplus \cdots \oplus M_{m}$$
 such that $dim_{M_{i}}=d$ and $\omega(M_{i},M_{j}) \ne 0$ only if $i+j=m+1$. We identify $M_{j}$ with $M^{*}_{m+1-j}$.
 
 Define the admissible parahoric subgroup $\mathbf{P}_{m} \subset G(K)$ to be the stabilizer of the lattice chain 
 $$\Lambda_{m} \supset \Lambda_{m-1} \supset \cdots \Lambda_{1}$$ where
 $$\Lambda_{i}=\sum_{1 \le j \le i}M_{j} \otimes \mathcal{O}_{K} + \sum_{i < j \le m}M_{j}\otimes \bar{\omega}\mathcal{O}_{K}$$
 and $\bar{\omega}$ is a uniformizer of $\mathcal{O}_{F}.$ It has Levi quotient denoted by $L_{m}=\prod_{i=1}^{\ell}GL(M_{i})$ where the $i$-th factor acts on $M_{i}$ by the standard representation and on $M^{*}_{i}=M_{m+1-i}$ by the dual of the standard representation. We have $\widetilde{L}_{m}=L_{m}$ and $\widetilde{L}^{ab}_{m} \cong \prod_{i=1}^{\ell}\mathbb{G}_{m}$ given by the determinants of the GL-factors. The vector space $V_{m}:=V_{\mathbf{P}_{m}}$ can be described as
$$ V_{m}=\operatorname{Sym}^{2}\left(M_{1}^{*}\right) \oplus \operatorname{Hom}\left(M_{2}, M_{1}\right) \oplus \cdots \oplus \operatorname{Hom}\left(M_{\ell}, M_{\ell-1}\right) \oplus \operatorname{Sym}^{2}\left(M_{\ell}\right).$$
We can arrange $M_{1}, \cdots, M_{m}$ into a cyclic quiver
$$
\begin{tikzcd}
  M_1 \arrow[r, "\phi_1", blue ] \arrow[d, "\psi_m" ] & M_2 \arrow[l, "\psi_1", shift left=1.5]  \arrow[r, blue ]& \cdots \arrow[l, shift left=1.5] \arrow[r,"\phi_{\ell -1}", blue ] & M_\ell \arrow[l, "\psi_{\ell-1}",shift left=1.5]  \arrow[d, "\phi_{\ell}",  shift left=1.5, blue ] \\
 M_m \arrow[r, "\psi_{m-1}"]  \arrow[u, "\phi_{m}",  shift left=1.5, blue ] & M_{m-1} \arrow[l, "\phi_{m-1}",  shift left=1.5, blue] \arrow[r] & \cdots \arrow[r, "\psi_{\ell +1}"]  \arrow[l, shift left=1.5, blue] & M_{\ell+1} \arrow[u, "\psi_{\ell}" ]  \arrow[l, "\phi_{\ell +1}",  shift left=1.5, blue]
\end{tikzcd}
$$
so that the involution $\tau$ sends $\left\{\psi_{i}: M_{i+1} \rightarrow M_{i}\right\}$ to $\left\{-\psi_{m-i}^{*}: M_{m-i}^{*} \rightarrow M_{m+1-i}^{*}\right\}$. Therefore, $V_{m}$ is the set of $\tau$-invariant cyclic quivers of the above shape. The dual space $V_{m}^{*}$ is the space of $\tau$-invariant cyclic quivers with all the arrows reversed, so we fit them within the same commutative diagram. Let $\phi_{i}: M_{i} \to M_{i+1}$ be the arrows and we view $\phi_{m}$ (resp. $\phi_{\ell}$) as a quadratic form on $M_{m}$ (resp. $M_{\ell}$). Then $\phi=\left(\phi_{1}, \ldots, \phi_{m}\right) \in V_{m}^{*}$ is stable if and only if
\begin{itemize}
	\item All the maps $\phi_{i}$ are isomorphisms;
\item We have two quadratic forms on $M_{m}: \phi_{m}$ and the transport of $\phi_{\ell}$ to $M_{m}$ using the isomorphism $\phi_{\ell-1} \cdots \phi_{1} \phi_{m}: M_{m} \stackrel{\sim}{\rightarrow} M_{\ell}$. They are in general position in the same sense as explained in \cite[Section 6.2]{YunEpipe}.
\end{itemize}
 
\subsection{The local system.}\label{symplectic local system}
Let $\mathcal{E}=M \otimes \mathcal{O}_X$ be the trivial vector bundle of rank $2 n$ over $X$ with a symplectic form (into $\mathcal{O}_X$ ) given by $\omega$. Define an increasing filtration of the fiber of $\mathcal{E}$ at $\infty$ by $F_{\le i} \mathcal{E}_{\infty}=\oplus_{j=1}^{i} M_{j}$ and a decreasing filtration of the fiber of $\mathcal{E}$ at 0 by $F^{\geq i} \mathcal{E}_{0}=\oplus_{j=i}^{m} M_{j}$. The moduli stack $\operatorname{Bun}_{G}(\widetilde{\mathbf{P}_{0}}, \mathbf{P}^{+}_{\infty})$ then classifies triples $(\mathcal{E}, F_{\le i} \mathcal{E}_{\infty}, F^{\geq i} \mathcal{E}_{0})$. The group ind-scheme $\mathfrak{G}$ is the group of symplectic automorphisms of $\left.\mathcal{E}\right|_{X-\{1\}}$ preserving the filtrations $F_{*}, F^{*}$ and acting by identity on the associated graded of $F_{*}$. Let $\lambda$ be the dominant short coroot. The subscheme $\mathfrak{G}_{\le \lambda}$ consists of those $g \in \mathfrak{G} \subset G(F)$ whose entries have at most simple poles at $t=1$, and $\operatorname{Res}_{t=1} g$ has rank at most one. \cite[Lemma 7.4]{YunEpipe} shows that the subscheme $\G_{\le \lambda}$ can be embedded as an open subscheme of $\operatorname{Sym}^{2}_{\le 1}(M)$, where $\operatorname{Sym}^{2}_{\le 1}(M) \subset \operatorname{Sym}^{2}(M)$ is the subscheme of \textit{symmetric pure $2$-tensors}. This is equivalent to say $u$ and $v$ are parallel vectors. We may write elements $u\cdot v \in \mathrm{Sym}^2(M)_{\leq 1}$ as $u=(u_1,\dots,u_m)$ and $v=(v_1,\dots,v_m)$ with $u_i, v_i\in M_i$. Define $\gamma_i$ as follows
\begin{align} \label{gamma}
    \gamma_i(u\cdot v):= \omega(v_{m+1-i}, u_i). 
\end{align}
The definition is independent of the choice of $u,v$ expressing $u\cdot v $ therefore defines a regular function on $\mathrm{Sym}^2(M)_{\leq 1}.$
The following proposition gives an explicit description of $\operatorname{Kl}_{\widehat{G},\bP}^{st}(\phi)$ that will be used later in the computation of the Euler characteristic.

 \begin{prop}\cite[Cor. 7.6]{YunEpipe}\label{Kl for symplectic}
     
Let $\phi=\left(\phi_{1}, \ldots, \phi_{m}\right) \in V_{m}^{*, \mathrm{st}}(k)$ be a stable functional. Recall that $\mathfrak{G}_{\le \lambda}$ in this case is $\operatorname{Sym}^{2}(M)_{\le 1}-\cup_{i=1}^{\ell} \Gamma_{i}$ where the divisor $\Gamma_{i}$ is defined by the equation $\gamma_{1}+\cdots+\gamma_{i}=1$ for functions $\gamma_{i}$ (see \cite[Prop. 7.5]{YunEpipe}). Let $f_{\phi}: X^{\circ} \times \mathfrak{G}_{\le \lambda} \rightarrow \mathbb{A}^{1}$ be given by
$$
f_{\phi}(x, u \cdot v)=\omega\left(\phi_{m} v_{m}, u_{m}\right) x+\sum_{i=1}^{\ell} \frac{\omega\left(\phi_{i} v_{i}, u_{m-i}\right)}{1-\gamma_{1}(u \cdot v)-\cdots-\gamma_{i}(u \cdot v)}
$$
Let $\pi: X^{\circ} \times \mathfrak{G}_{\le \lambda} \rightarrow X^{\circ}$ be the projection. Then we have an isomorphism over $X^{\circ}$
$$
\mathrm{Kl}_{\widehat{G}, \mathbf{P}_{m}}{\mathrm{St}}(\mathbf{1}, \phi) \cong \pi ! f_{\phi}^{*} \mathrm{AS}_{\psi}[2n-1]\left(\frac{2 n-1}{2}\right).
$$
  \end{prop}

\subsection{Computation of Euler characteristic}
\begin{thm}\label{symplectic}
	We have 
 \begin{align}\label{symp-equation}
     -\chi_{c}(\widetilde{X}^{\circ}, \mathrm{Kl}^{\operatorname{st}}_{\widehat{G}, \mathbf{P}_{m}}(\chi , \phi))= d 
 \end{align}

\end{thm} 
By the same argument as in \cite{Katz} the Swan conductor of $ \mathrm{Kl}^{\operatorname{st}}_{\widehat{G}, \mathbf{P}_{m}}(\chi , \phi)$ at $\infty$ does not depend on $\chi$, so does the Euler characteristic of $\mathrm{Kl}^{\operatorname{st}}_{\widehat{G}, \mathbf{P}_{m}}(\chi , \phi)$. We make an inductive argument as follows.
 \begin{prop}\label{4.3}
Let $m \ge 3$ be a regular elliptic number. Let $\Gamma_i \subset \mathrm{Sym}^2_{\leq 1}(M)$ be the divisor defined by the equation $\gamma_i'=\gamma_1+\cdots+\gamma_i=1$ for functions $\gamma_i$ (recall \ref{gamma}). Let $U_i = \mathrm{Sym}^2_{\leq 1}(\oplus_{j=1}^i M_j \oplus_{j'=m-i-1}^{m} M_{j'}) - \cup_{j=1}^{i}\Gamma_i$ for $1\leq i \leq \ell-1$, $U_{\ell}= \mathrm{Sym}^2_{\leq 1}(M) - \cup_{j=1}^{\ell}\Gamma_i$, and $U_0 = \mathrm{Sym}^2_{\leq 1}(M_1\oplus M_m) - \Gamma_1$. Let $W_i \subset U_i$ be the divisor defined by $\omega(\phi_{m}(v_m), u_m)=0$ for $0\leq i \leq \ell$. Assume $\chi =1$. We have
\begin{align}\label{split}
    (-1)^{2n-1}\chi_c(\widetilde{X}^{\circ}, \mathrm{Kl}^{\operatorname{st}}_{\widehat{G}, \mathbf{P}_{m}}(\chi , \phi)) = -\chi_c(\mathrm{Sym}_{\le 1}^2(M_1 \oplus M_{m}))+\chi_c(\Gamma_1)+\chi_c(\Gamma_1')-\chi_c(\Gamma_1 \cap \Gamma_1'). 
\end{align}
 \end{prop}
\begin{proof}
 We assume $\mathcal{K}$ is the trivial character sheaf as in \cite[Prop 6.8]{YunEpipe}. 
 By definition, $U_{\ell}=\mathfrak{G}_{\le \lambda}$.
Let $f_i: U_{\ell} \to \mathbb{A}^1$ be the function $[v] \mapsto \frac{\omega(\phi_{i}v_{i}, u_{m-i})}{1-\gamma_i'(u\cdot v)}$. This function only depends on the coordinates $u_1, \cdots, u_i, u_{m-i},\cdots, u_m, v_1, \cdots, v_i, v_{m-i},\cdots, v_m$. Let $f_{\leq i}=f_1+\cdots f_i$. Let $f_m=\omega(\phi_mv_m, u_m)$. 

Consider the projection $\pi_2: \mathbb{G}_m^{\mathrm{rot}} \times \mathfrak{G}_{\lambda} \to \mathfrak{G}_{\lambda}$. The stalk of $\pi_{2,!}f_{\phi}^{\ast}\mathrm{AS}_{\psi}$ over $u\cdot v$ is 
\[
f^{\ast}_{\leq \ell}\mathrm{AS}_{\psi} \otimes \mathrm{H}^{\ast}_c(\mathbb{G}_m^{\mathrm{rot}}, T^{\ast}_{f_{m}(u\cdot v)}\mathrm{AS}_{\psi}) 
\]
where $T_{f_{m}(u\cdot v)}$ is the map $\mathbb{G}_m \to \mathbb{A}^1$ given by multiplication by $f_{m}(u\cdot v)$. 

When $f_{m}(u\cdot v)=0$ we have $\mathrm{H}^{\ast}_c (\mathbb{G}_m^{\mathrm{rot}}, T^{\ast}_{f_{m}(u\cdot v)}\mathrm{AS}_{\psi})=\mathrm{H}^{\ast}_c (\mathbb{G}_m, 0^{\ast}\mathrm{AS}_{\psi})=\mathrm{H}^{\ast}_c(\mathbb{G}_m, \overline{\mathbb{Q}_{\ell}})$. When $f_{m}(u\cdot v)\neq 0$, since $H_c^{\ast}(\mathbb{A}^1, \mathrm{AS}_{\psi})=0$, we have $\mathrm{H}^{\ast}_c (\mathbb{G}_m^{\mathrm{rot}}, T^{\ast}_{f_{m}(u\cdot v)}\mathrm{AS}_{\psi})=\overline{\mathbb{Q}_{\ell}}[-1]$. Therefore we have
\begin{align}\begin{split}\label{chiind}
(-1)^{2n-1}\chi_c(\widetilde{X}^{\circ}, \mathrm{Kl}^{\operatorname{st}}_{\widehat{G}, \mathbf{P}_{m}}(1 , \phi)) &= \chi_c(\mathfrak
    {G}_{\leq \lambda}, \pi_{2!}f_{\phi}^{\ast}\mathrm{AS}_{\psi}) \\
    &=  -\chi_c(U_{\ell}, f^{\ast}_{\leq \ell}\mathrm{AS}_{\psi})+\chi_{c}(W_{\ell}, f^{\ast}_{\leq \ell}\mathrm{AS}_{\psi})
    \end{split}
\end{align}
We will show that the equalities
\begin{align}\label{Uind}
    \chi_c(U_{i+1}, f^{\ast}_{\leq i+1} \mathrm{AS}_{\psi}) =  \chi_c(U_{i}, f^{\ast}_{\leq i} \mathrm{AS}_{\psi})
\end{align}
and 
\begin{align}\label{Wind}
    \chi_c(W_{i+1}, f^{\ast}_{\leq i+1} \mathrm{AS}_{\psi}) =  \chi_c(W_{i}, f^{\ast}_{\leq i} \mathrm{AS}_{\psi})
\end{align}
holds for $1 \leq i \leq \ell-1$.
Consider the projection $p: U_{i+1} \to U_i'$ by forgetting the $M_{m-i-2}$ component. We have
\[
p_{!}f_{\leq i+1}^{\ast}\mathrm{AS}_{\psi}  = f^{\ast}_{\leq i+1} \mathrm{AS}_{\psi}   \otimes p_{!}\overline{\mathbb{Q}_{\ell}}.
\]
Hence 
\[
\chi_c(U_{i+1}, f_{\leq i+1}^{\ast}\mathrm{AS}_{\psi}) = \chi_c(U_i', f_{\leq i+1}^{\ast}\mathrm{AS}_{\psi}\otimes p_{!}\overline{\mathbb{Q}_{\ell}})
\]
Fix $u'\cdot v'\in U_i'$ such that $$u'=(u_1',\cdots,u_{i+1}',u_{m-i-1}',\cdots,u_m'),$$ 
$$v'=(v_1',\cdots,v_{i+1}',v_{m-i-1}',\cdots,v_m'),$$ and let 
\[
\gamma_i'=\gamma_1(u'\cdot v')+\cdots+\gamma_i(u'\cdot v') = \omega(v_{m}',u_1')+\cdots+\omega(v_{m+1-i}',u_{i}).
\]
The fiber of $p$ over $u'\cdot v'$ is $M_{m-i-2}$. Thus we have $H_c^{\ast}(p^{-1}(u'\cdot v'))=\overline{\mathbb{Q}_{\ell}}[-2d]$. Moreover, $$\chi_c(U_{i+1}, f_{\leq i+1}^{\ast}\mathrm{AS}_{\psi})=\chi_c(U_i', f^{\ast}_{\leq i+1} \mathrm{AS}_{\psi}).$$ Since $U_i$ can be identified with the subscheme of $U_i'$ where $v_{i+1}=0$, we have $\chi_c(U_i', f^{\ast}_{\leq i+1} \mathrm{AS}_{\psi}) = \chi_c(U_i, f^{\ast}_{\leq i}\mathrm{AS}_{\psi})$. 

Recall that $U_0 = \mathrm{Sym}^2_{\leq 1}(M_1\oplus M_m) - \Gamma_1$. Consider the projection $p_0: U_1 \to U_0$ by forgetting the $M_{m-2}, M_{m-1}$ component. We have
\[
\chi_c(U_1, f_1^{\ast}\mathrm{AS}_{\psi})=\chi_c(U_0, 0^{\ast}\mathrm{AS}_{\psi}\otimes p_{0!}f_1^{\ast}\mathrm{AS}_{\psi})
\]
Fix $u\cdot v \in U_0$ where $u=(u_1,u_m)$ and $v=(v_1,v_m)$. By definition, $\gamma_1=\omega(v_m,u_1)$. The fiber of $p_0$ is $M_{m-2}\otimes M_{m-1}$, therefore $H_c^{\ast}(p_0^{-1}(u\cdot v))=\overline{\mathbb{Q}_{\ell}}[-4d]$. One deduce that 
\begin{align}\label{U0ind}
    \chi_c(U_1, f_1^{\ast}\mathrm{AS}_{\psi})=\chi_c(U_0).
\end{align}

Consider the projection $p: W_{i+1} \to W_i$. We have
\[
p_{!}f_{\leq i+1}^{\ast}\mathrm{AS}_{\psi}  = f^{\ast}_{\leq i} \mathrm{AS}_{\psi}   \otimes p_{!}f_{i+1}^{\ast} \mathrm{AS}_{\psi}.
\]
Fix $u'\cdot v' \in W_i$. The fiber of $p^{-1}(u'\cdot v')$ is $\{(u_{i+1}, u_{m-i-2}) \in M_{i+1} \times M_{m-i-2}) | \omega(v_{m-i}',u_{i+1})+\Gamma_i \neq 1 \}$. The function $f_{i+1}$ along the fiber $p^{-1}(u'\cdot v')$ is a linear function in $v_{i+1}$ (the vector parallel to $u_{i+1}$) given by $f_{i+1}=\frac{\omega(\phi_{i+1}v_{i+1}, u_{m-i-1})}{1-\omega(v'_{m-i},u_{i+1})-\Gamma_i}$. The stalk of $p_{1,!}f_{i+1}^{\ast}\mathrm{AS}_{\psi}$ at $u'\cdot v'$ is $H_c^{\ast}(p^{-1}(u'\cdot v'), f_{i+1}^{\ast}\mathrm{AS}_{\psi})$, which vanishes when $u'_{m-i-1}\neq 0$. We see that
\[
H_c^{\ast}(p^{-1}(u'\cdot v'), f_{i+1}^{\ast}\mathrm{AS}_{\psi}) = H_c^{\ast}(p^{-1}(u'\cdot v'), 0^{\ast}\mathrm{AS}_{\psi}) = H_c^{\ast}(p^{-1}(u'\cdot v'), \overline{\mathbb{Q}_{\ell}}).
\]
When $v'_{m-i} = 0$ this term equals $\overline{\mathbb{Q}_{\ell}}[-4d]$ and when $v'_{m-i} \neq 0$, it is $H_c^{\ast}(\mathbb{G}_m)[-4d+2]$. Therefore, $p_{!}f_{\leq i+1}^{\ast}\mathrm{AS}_{\psi}$ and the constant sheaf $\overline{\mathbb{Q}_{\ell}}$ are the same in the Grothendieck group of $D_c^b(W_i)$. Thus we have 
\[
\chi_c(W_{i+1}, f_{\leq i+1}^{\ast}\mathrm{AS}_{\psi}) = \chi_c(W_i, f^{\ast}_{\leq i} \mathrm{AS}_{\psi}   \otimes p_{!}f_{i+1}^{\ast} \mathrm{AS}_{\psi}) = \chi_c(W_i, f^{\ast}_{\leq i} \mathrm{AS}_{\psi})
\]
for $1\leq i \leq \ell-1$. Similar as in the $U_0$ case, we deduce that
\begin{align}\label{W0ind}
    \chi_c(W_1, f_1^{\ast}\mathrm{AS}_{\psi}) = \chi_c(W_0). 
\end{align}

Therefore by \cref{chiind}, \cref{Uind}, \cref{U0ind}, and \cref{Wind}, \cref{W0ind}, we have
\[
(-1)^{2n-1}\chi_c(\widetilde{X}^{\circ}, \mathrm{Kl}^{\operatorname{st}}_{\widehat{G}, \mathbf{P}_{m}}(\chi , \phi)) = -\chi_c(U_0)+\chi_c(W_0),
\]
where $U_0=\mathrm{Sym}^2_{\leq 1}(M_1 \oplus M_{m}) -  \Gamma_1$, $W_0 = \Gamma_1' - \Gamma_1$ and $\Gamma_1'$ is the divisor defined by $\omega(\phi_m(v_m), u_m)=0$. Hence 
\begin{align}
    (-1)^{2n-1}\chi_c(\widetilde{X}^{\circ}, \mathrm{Kl}^{\operatorname{st}}_{\widehat{G}, \mathbf{P}_{m}}(\chi , \phi)) = -\chi_c(\mathrm{Sym}_{\le 1}^2(M_1 \oplus M_{m}))+\chi_c(\Gamma_1)+\chi_c(\Gamma_1')-\chi_c(\Gamma_1 \cap \Gamma_1'). 
\end{align}

\end{proof}

Now we consider the case when $m=2$. Let $\widetilde{U}_0:=U_0=\mathrm{Sym}_{\leq 1}^{2}(M_1 \oplus M_2) - \Gamma_1$ and let $\widetilde{W}_0 \subset \widetilde{U}_0$ be the divisor defined by $\omega(\phi_2v_2,u_2)=0$. We prove the following proposition.

\begin{prop}\label{4.4}
  Let $\widetilde{U}_0' \subset \widetilde{U}_0$ (resp. $W_0' \subset W_0$) be the divisor defined by $\omega(\phi_1v_1,u_1) = 0$  and let $\widetilde{U}=\widetilde{U}_0 - \widetilde{U}_0'$ (resp. $\widetilde{W}=\widetilde{W}_0-\widetilde{W}_0'$). Let $\mathbb{P}(\widetilde{U}) \subset \mathbb{P}(\mathrm{Sym}_{\le 1}^2(M_1\oplus M_2))$ (resp. $\mathbb{P}(\widetilde{W})$) be the projectivization of $\widetilde{U}$ so that $u\cdot ku$ is equivalent to $u \cdot u$ for $u\cdot ku \in \widetilde{U}$.   We have 
    \begin{align*}
     -\chi_c(\widetilde{X}^{\circ}, \mathrm{Kl}^{\operatorname{st}}_{\widehat{G}, \mathbf{P}_{2}}(1, \phi)) = \chi_c(\mathbb{P}(\widetilde{U})) - \chi_c(\mathbb{P}(\widetilde{W}))- \chi_c(\widetilde{U}_0)+\chi_c(\widetilde{W}_0).
\end{align*}
\end{prop}
\begin{proof}
    we have
\begin{align*}\label{m=2 chi}
    -\chi_c(\widetilde{X}^{\circ}, \mathrm{Kl}^{\operatorname{st}}_{\widehat{G}, \mathbf{P}_{2}}(1, \phi)) &= -\chi_c(\widetilde{U}_0,f_1^{\ast}\mathrm{AS}_{\psi})+\chi_c(\widetilde{W}_0,f_1^{\ast}\mathrm{AS}_{\psi}) \\
 &= - \chi_c(\widetilde{U},f_1^{\ast}\mathrm{AS}_{\psi})+\chi_c(\widetilde{W},f_1^{\ast}\mathrm{AS}_{\psi}) - \chi_c(\widetilde{U}_0,f_1^{\ast}\mathrm{AS}_{\psi})+\chi_c(\widetilde{W}_0,f_1^{\ast}\mathrm{AS}_{\psi})  \\
    &=  - \chi_c(\widetilde{U},f_1^{\ast}\mathrm{AS}_{\psi})+\chi_c(\widetilde{W},f_1^{\ast}\mathrm{AS}_{\psi}) - \chi_c(\widetilde{U}_0,0^{\ast}\mathrm{AS}_{\psi})+\chi_c(\widetilde{W}_0,0^{\ast}\mathrm{AS}_{\psi})  \\
   &= -\chi_c(\widetilde{U},f_1^{\ast}\mathrm{AS}_{\psi})+\chi_c(\widetilde{W},f_1^{\ast}\mathrm{AS}_{\psi}) - \chi_c(\widetilde{U}_0)+\chi_c(\widetilde{W}_0) 
\end{align*}

By change of variable isomorphism and K$\ddot{\mathrm{u}}$nneth formula, we have
\begin{align*}
& -\chi_c(\widetilde{U},f_1^{\ast}\mathrm{AS}_{\psi})+\chi_c(\widetilde{W},f_1^{\ast}\mathrm{AS}_{\psi}) \\
=& -\chi_c(\mathbb{G}_m, \mathrm{AS}_{\psi}) \chi_c(\mathbb{P}(\widetilde{U}))+ \chi_c(\mathbb{G}_m, \mathrm{AS}_{\psi}) \chi_c(\mathbb{P}(\widetilde{W}))\\
=&  \chi_c(\mathbb{P}(\widetilde{U})) - \chi_c(\mathbb{P}(\widetilde{W}))
\end{align*}
since $\chi_c(\mathbb{G}_m, \mathrm{AS}_{\psi})=-1$. 
The proposition follows.
\end{proof}

\vspace{1em}

The proof of Theorem \ref{symplectic} then reduces to the calculation of the compactly supported Euler characteristic of each pieces in Proposition \ref{4.3} and \ref{4.4}, therefore boils down to the proof of a few lemmas listed below. 
  \begin{proof}[Proof of Theorem \ref{symplectic}]

  Let $m$ be a regular elliptic number.
   
\noindent \textbf{Assume m $\ne$ 2.}    We first notice that $\mathrm{Sym}_{\le 1}^2(M_1 \oplus M_{m})$ is a cone over $\bP(M_1 \oplus M_{m})$, i.e. a line bundle $\mathcal{L}$ over $M_1 \oplus M_{m}$ minus the zero section. Let $E$ be the total space of the line bundle and denote by $E_0$ the zero section. Let $E_c$ be the complement of the disk bundle $D(\mathcal{L})$.

\begin{lem}\label{echar sym square}
    $$\chi_{c}(\mathrm{Sym}_{\le 1}^2(M_1 \oplus M_{m}))=1.$$

\end{lem}
\begin{proof} We have
    \begin{align*}
        H_{\star}(E/E_0) =H_{\star}(E/E_0, E_c) = H_{\star}(E/E_c, E_0) 
    \end{align*}
      where $E/E_c$ is the Thom space of $\mathcal{L}$. Then by Thom isomorphism 
      $$H_{\star +2}(E/E_c, pt) \cong H_{\star}(M_1 \oplus M_{m})$$
  The result follows from the Poincare duality.
\end{proof}

 \begin{lem}
$$\chi_c(\Gamma_1)=0.$$
\end{lem}
\begin{proof}
Since $\Gamma_1$ is defined by the symplectic form $\omega(u_1, v_m)=1$, $v_m$ is determined by the choice of $u_1$ up to a non-zero scalar. Therefore we have the fibration
$$\{v_m \mid \omega(u_1, v_m )=1\} \to \Gamma_1 \to M_1 \backslash \{0\}$$
where the fiber $\{v_m \mid \omega(u_1, v_m )=1\}$ contains a summand $k^*.$
The fibration property of the Euler characteristic implies that $$\chi_c(\Gamma_1)=\chi_c(k^*) \cdot \chi_c(M_1 \backslash \{0\})=0.$$
\end{proof}

\begin{lem}\label{quadric on M_1 + M_m}
$$\chi_c( \Gamma_1')=1.$$
\end{lem}
\begin{proof}

Let $\mathbb{P}(\Gamma_1')$ be the projectivization of $\Gamma_1'$, it is a projective quadric defined by the same homogenous equation as $\Gamma_1'$. For $\Gamma_1' - \{0\}$, we have the following fibration
$$\{(k_1, u, v_m)\} \to \Gamma_1' -\{0\} \to \mathbb{P}(\Gamma_1')$$ over a fixed $[v_1] \in \mathbb{P}(\Gamma_1') \subset \mathbb{P}(M_1),$
where $k_1 v_1$ is the preimage of $[v_1].$
But it is still a $k^*$-bundle, therefore $\chi_c(\Gamma_1' -\{0\})=0.$

The fiber over $v_1=0$ is nothing but $\mathrm{Sym}_{\le 1}^2(M_m)$. The lemma then follows from Lemma \ref{echar sym square}.

\end{proof}

\begin{lem}\label{two quadrics}
$$\chi_c(\Gamma_1 \cap \Gamma_1')=-d.$$
\end{lem}
\begin{proof}
Let $X=\Gamma_1 \cap \Gamma_1'.$ Then 
\begin{align*}
  \chi_c(X) & =\chi_c(\Gamma_1' \cap \operatorname{Sym}^2_{\le 1})- \chi_{c}(\Gamma_1' \cap \omega(v_1, u_m)=0))  \\
  &= \chi_c(\Gamma_1') - \chi_c( \omega(\phi_m(v_m), u_m) = 0 \cap (\omega(v_1, u_m)=0)) \\
\end{align*}

Let $\Gamma_0:=\omega(\phi_m(v_m), u_m) = 0 \cap (\omega(v_1, u_m)=0)$. Then $\Gamma_0$ minus $V(v_1=0)$ is a bundle over $\mathbb{P}(M_1)$. Consider its projection 
 $$\psi: \Gamma_0 - V(v_1=0) \to \mathbb{P}(M_1) $$
 $$(v, u) \mapsto [v_1].$$

 The fiber over each $[v_1]$ is a cone over $\mathbb{P}(M)$, i.e., the quadric $\omega(\phi_m(v_m), u_m)$ intersecting a hyperplane, therefore an affine quadric with one less variable. We have $$\chi_c(  \Gamma_0 - V(v_1=0))=1 \cdot \chi_c(\mathbb{P}(M_1))=d.$$

Assume $v_1=0$. The intersection then becomes a single quadric $\omega(\phi_m(v_m), u_m)=0$ in $\operatorname{Sym}^{2}_{\le 1}(M_m)$, whose compactly supported Euler characteristic is $1$ by Lemma \ref{quadric on M_1 + M_m}. Putting things together, one easily deduces the lemma.

\end{proof}

Therefore \ref{symp-equation} holds for $m \ge 3$.

\noindent \textbf{Assume m=2.} Throughout the proof, we make the following shorthand  
\begin{itemize}
\item $\tilde{U}_0':=Q_1 -\Gamma_1$, where $Q_1$ is the quadric defined by $\omega(\phi_1 v_1, u_1);$
    \item $\tilde{W}_0:= Q_2 -\Gamma_1$, where $Q_2$ is the quadric defined by $\omega(\phi_2 v_2, u_2).$ 
\end{itemize}We deduce from Lemma \ref{echar sym square}-\ref{two quadrics} and (\ref{split}) that
$$-\chi_c(U_0) + \chi_c(W_0)=-\chi_c(\mathrm{Sym}_{\le 1}^2(M_1 \oplus M_{2}))+\chi_c(\Gamma_1)+\chi_c(Q_2)-\chi_c(\Gamma_1 \cap Q_2) =d.$$ The theorem follows from the following proposition.
  \begin{prop}
 	$$\chi_c(\mathbb{P}(U)) -\chi_c(\mathbb{P}(W))=0.$$
 \end{prop}

\begin{proof}
We have $$\chi_c(\mathbb{P}(U)) -\chi_c(\mathbb{P}(W))=\chi_c(\mathbb{P}(U_0)) - \chi_c(\mathbb{P}(U_0')) - \chi_c(\mathbb{P}(W_0)) + \chi_c(\mathbb{P}(W_0 \cap U_0'))$$
 
so it is sufficient to do the computation term by term. First,
\begin{align}
	\chi_c(\mathbb{P}(U_0)) &=\chi_c(\mathbb{P}(M_1 \oplus M_2))- \chi_c(\mathbb{P}(\Gamma_1)) \\
	&=2d-(2d-2d) \\
	&=2d
\end{align}

Then we notice that
\begin{align}
	\chi_c(\mathbb{P}(U_0')) &=\chi_c(\mathbb{P}(Q_1)- \chi_c(\mathbb{P}(Q_1 \cap \Gamma_1)) .
	\end{align}
	By Lemma \ref{Q_ell}, $\chi_c(\mathbb{P}(Q_1))=2d$ (when $d$ is even, $2d-1$ when $d$ is odd). There are multiple ways to compute $\chi_c(\mathbb{P}(Q_1 \cap \Gamma_1))$. For example one can take the projection onto the quadric $\mathbb{P}(U_0')$ by sending $[v_1, v_2]$ to $[v_1]$. The fiber $F_{v_1}$ is $M_2$ minus a hyperplane, therefore has Euler characteristic $\chi_c(F_{v_1})=1-1=0$. The fibration property implies $\chi_c(\mathbb{P}(Q_1 \cap \Gamma_1)=0$, so $\chi_c(\mathbb{P}(U_0'))=2d$(when $d$ is even, $2d-1$ when $d$ is odd).

	We have $\chi_c(\mathbb{P}(W_0))=2d$ (when $d$ is even, $2d-1$ when $d$ is odd) by the same reason.
	
	The Euler characteristic of $\chi_c(\mathbb{P}(W_0 \cap U_0'))$ can be computed as follows, combining with Lemma \ref{Q_ell} and Lemma \ref{three quadrics}.
	\begin{align*}
		\chi_c(\mathbb{P}(W_0 \cap U_0')) &= \chi_c(\mathbb{P}(Q_1 \cap Q_2)) - \chi_c(\mathbb{P}(\Gamma_1 \cap Q_1 \cap Q_2)) \\
		&= \chi_c(\mathbb{P}(Q_1 \cap Q_2)) -[\chi_c(\mathbb{P}(Q_1 \cap Q_2))-\chi_c(\Gamma^{\perp}_1 \cap Q_1 \cap Q_2)]\\
				&= 2d, \text{   if } d \text{ is even }  \\
		&=2d-2, \text{   if } d \text{ is odd }
\end{align*}
where $\Gamma^{\perp}_1$ is the quadric in $\mathbb{P}(M_1 \oplus M_2)$ defined by $\omega(v_1, v_2)=0.$ One then deduces the lemma by putting together.
\end{proof}

\end{proof}

\vspace{0.5in}

\section{Euler Characteristic for Split and Quasi-split Orthogonal Groups} \label{Sec5}
The goal of this section is to prove Theorem \ref{Echar orthogonal}. Throughout the section we work over $\bar{k}$ and ignore all the Tate twists. Let $G$ be a split or quasi-split orthogonal group over $K$. 

\subsection{The set-up} 
Assume $\mathrm{char}(k)\neq 2$. Let $(M, q)$ be a quadratic space of dimension $2n$ or $2n+1$ over k. Let $(\cdot,\cdot)$: $M\times M \to k$ be the associated symmetric bilinear form $(x,y)=q(x+y)-q(x)-q(y)$. The regular elliptic numbers of $m$ of the root systems of type $B_n, D_n$ and $^{2}D_n$ are in bijection with 
\begin{itemize}
    \item Type $B_n(\operatorname{dim} M=2 n+1)$ : divisors $d \mid n$ (corresponding $\left.m=2 n / d\right)$;
\item Type $D_n(\operatorname{dim} M=2 n)$ : even divisors $d \mid n$ (corresponding $\left.m=2 n / d\right)$ or odd divisors $d \mid n-1$ (corresponding $m=2(n-1) / d$ );
\item Type ${ }^2 D_n(\operatorname{dim} M=2 n)$ : odd divisors $d \mid n$ (corresponding $\left.m=2 n / d\right)$ or even divisors $d \mid n-1$ (corresponding $m=2(n-1) / d$ ).
\end{itemize}
Let $m=2 \ell$ for some $\ell$. Fix a decomposition
\begin{align}\label{5.1}
    M=M_0 \oplus M_1 \oplus \cdots M_{\ell-1} \oplus M_{\ell} \oplus M_{\ell+1} \oplus \cdots \oplus M_{m-1}
\end{align}

where $\operatorname{dim} M_i=d$ for $i=1, \cdots, \ell-1, \ell+1, \cdots, m-1$, $\operatorname{dim} M_0$ and $\operatorname{dim} M_{\ell}$ are either $d$ or $d+1$, and we make sure that when $\operatorname{dim} M=2 n+1$, $\operatorname{dim} M_0$ is even. We see that

\begin{itemize}
\item Type $B_n: \operatorname{dim} M_0$ is even and $\operatorname{dim} M_{\ell}$ is odd;
\item Type $D_n: \operatorname{dim} M_0=\operatorname{dim} M_{\ell}$ is even;
\item Type ${ }^2 D_n: \operatorname{dim} M_0=\operatorname{dim} M_{\ell}$ is odd.
\end{itemize}

The decomposition \ref{5.1} satisfy $\left(M_i, M_j\right)=0$ unless $i+j \equiv 0 \bmod m$. The restriction of $q$ to $M_0$ and $M_{\ell}$ are denoted by $q_0$ and $q_{\ell}$. The pairing $(\cdot, \cdot)$ induce an isomorphism $M_i^* \cong M_{m-i}$.
Let $M_{+}=\oplus_{i>0} M_i$, then $M=M_0 \oplus M_{+}$ and correspondingly $q=q_0 \oplus q_{+}$.

 Let $\widetilde{\mathbf{P}}_{m} \subset G(K)$ be the stabilizer of the lattice chain 
 $$\Lambda_{m} \supset \Lambda_{m-1} \supset \cdots  \supset \Lambda_{0}$$ where
 $$\Lambda_{i}=\sum_{0 \le j \le i}M_{j} \otimes \mathcal{O}_{K} + \sum_{i < j \le m-1}M_{j}\otimes \bar{\omega}\mathcal{O}_{K}$$
 where $\bar{\omega}$ is a uniformizer of $\mathcal{O}_{F}.$ Its reductive quotient $\widetilde{L}_m$ is the subgroup of $\mathrm{O}\left(M_0, q_0\right) \times \prod_{i=1}^{\ell-1} \mathrm{GL}(M_i) \times \mathrm{O}\left(M_{\ell}, q_{\ell}\right)$ of index two consisting of $\left(g_0, \cdots, g_{\ell}\right)$ where $\operatorname{det}\left(g_0\right)=\operatorname{det}\left(g_{\ell}\right)$. 

 The subgroup $\mathbf{P}_m \subset \widetilde{\mathbf{P}}_m$, defined as the kernel of $\widetilde{\mathbf{P}}_m \rightarrow \widetilde{L}_m \rightarrow\{ \pm 1\}$ by taking the determinant of the first factor, is an admissible parahoric subgroup of $G(K)$ with $m\left(\mathbf{P}_m\right)=m$. The vector space $V_m:=V_{\mathbf{P}_m}$ is
$$
V_m=\operatorname{Hom}\left(M_1, M_0\right) \oplus \operatorname{Hom}\left(M_2, M_1\right) \oplus \cdots \oplus \operatorname{Hom}\left(M_{\ell}, M_{\ell-1}\right) .
$$
And similarly, we can arrange $V_{m}$ into a cyclic quiver
$$
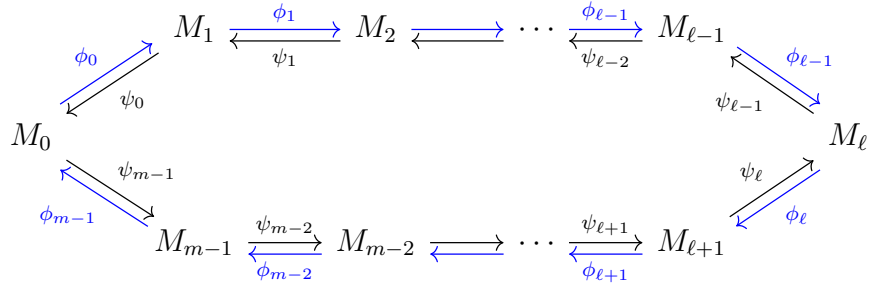
\begin{tikzcd}
 & M_1 \arrow[r, "\phi_1", blue ] \arrow[dl, "\psi_0" ] & M_2 \arrow[l, "\psi_1", shift left=1.5]  \arrow[r, blue ]& \cdots \arrow[l, shift left=1.5] \arrow[r,"\phi_{\ell -1}", blue ] & M_{\ell-1} \arrow[l, "\psi_{\ell-2}",shift left=1.5]  \arrow[dr, "\phi_{\ell-1}",  shift left=1.5, blue ] & \\
  M_0 \arrow[dr, "\psi_{m-1}"]  \arrow[ur, "\phi_{0}", shift left=1.5, blue ] & & & &  & M_{\ell} \arrow[ul, "\psi_{\ell -1}"] \arrow[dl, "\phi_{\ell}",shift left=1.5, blue] \\
 & M_{m-1} \arrow[r, "\psi_{m-2}"]  \arrow[ul, "\phi_{m-1}",  shift left=1.5, blue ] & M_{m-2} \arrow[l, "\phi_{m-2}",  shift left=1.5, blue] \arrow[r] & \cdots \arrow[r, "\psi_{\ell +1}"]  \arrow[l, shift left=1.5, blue] & M_{\ell+1} \arrow[ur, "\psi_{\ell}"  ]  \arrow[l, "\phi_{\ell +1}",  shift left=1.5, blue] &
\end{tikzcd}
$$
where the involution $\tau$ sends $\left\{\psi_{i}: M_{i+1} \rightarrow M_{i}\right\}$ to $\left\{-\psi_{m-1-i}^{*}:M_{m-1-i}^{*} \rightarrow M_{m+1-i}^{*}\right\}$. The dual space $V_{m}^{*}$ is the space of $\tau$-invariant cyclic quivers with all the arrows reversed. Let $\phi_{i}: M_{i} \to M_{i+1}$ be the arrows. Then $\phi=\left(\phi_{0}, \ldots, \phi_{m-1}\right) \in V_{m}^{*}$ is stable if and only if
\begin{itemize}
	\item All the maps $\phi_i$ have the maximal possible rank;
\item  We have two quadratic forms on $M_0: q_0$ and the pullback of $q_{\ell}$ to $M_0$ via the map $\phi_{\ell-1} \cdots \phi_0: M_0 \rightarrow M_{\ell}$. They are in general position in the same sense as the symplectic and unitary cases.
\end{itemize}
 
\subsection{The local system.}\label{orthogonal local system}
The moduli stack $\operatorname{Bun}_{G}(\widetilde{\mathbf{P}_{0}}, \mathbf{P}^{+}_{\infty})$ classifies 4-tuples $$(\mathcal{E}, \mathcal{E}(-\{\infty \}),\mathcal{E}(-\{0 \}), \delta),$$ where the vector bundle $\mathcal{E}$, an increasing filtration $\mathcal{E}(-\{\infty \})$ of the fiber of $\mathcal{E}$ at $\infty$ , a decreasing filtration $\mathcal{E}(-\{0 \})$ at $0$ and a trivialization $\delta$ of $\mathcal{O}_X$ are defined in \cite[Sec. 8.3]{YunEpipe}. The group ind-scheme $\mathfrak{G}$ is the group of orthogonal automorphisms of $\left.\mathcal{E}\right|_{X-\{1\}}$ preserving all the auxiliary data specified above. Let $\lambda \in \mathbb{X}_{*}(T)$ be the dominant minuscule coweight such that $V_{\lambda}$ is the standard representation of the dual group $\widehat{G}=\operatorname{Sp}_{2n}$ or $\operatorname{SO}_{2n}$. The subscheme $\mathfrak{G}_{\le \lambda}$ consists of those $g \in \mathfrak{G} \subset G(F)$ whose entries have at most simple poles at $t=1$, and $\operatorname{Res}_{t=1} $g has rank one. \cite[Lem. 8.5 (3)]{YunEpipe} shows that the subscheme $\G_{\le \lambda}$ can be embedded as an open subscheme of the quadric $Q(q)$ in $\mathbb{P}(M)$ defined by $q=0$. Let $q_{[i,m-i]}$ be the restriction of the quadratic form $q$ to $M_{i}\oplus\cdots M_{m-i}$ that extended to $M$ by zero on the rest of the direct summands. Similar to Proposition \ref{Kl for symplectic} in the case of symplectic groups, the following proposition gives an explicit description of $\operatorname{Kl}_{\widehat{G},\bP}^{st}(\phi)$ when $G$ is split or quasi-split orthogonal.
\begin{prop}\cite[Cor. 8.7]{YunEpipe}\label{Kl for orthogonal}
 Let $\phi=\left(\phi_0, \phi_1, \ldots, \phi_{m-1}\right) \in V_m^{* \text { st }}(k)$ be a stable functional. Recall that $\mathfrak{G}_\lambda$ in this case is $Q(q)-\cup_{i=1}^{\ell} Q\left(q_{[i, m-i]}\right)$. Let $f_\phi: \widetilde{X}^{\circ} \times \mathfrak{G}_\lambda \rightarrow \mathbb{A}^1$ be given by
$$
f_\phi(x,[v])=-\frac{\left(\phi_0 v_0, v_{m-1}\right)}{q_{[1, m-1]}(v)} x-\sum_{i=1}^{\ell-1} \frac{\left(\phi_i v_i, v_{m-i-1}\right)}{q_{[i+1, m-i-1]}(v)} .
$$

Let $\pi: \widetilde{X}^{\circ} \times \mathfrak{G}_\lambda \rightarrow \widetilde{X}^{\circ}$ be the projection. Then we have an isomorphism over $\widetilde{X}^{\circ}$
$$
\mathrm{Kl}_{\widehat{G}, m}^{\mathrm{St}}(\mathbf{1}, \phi) \cong \pi_{!} f_\phi^* \mathrm{AS}_\psi[\operatorname{dim} M-2]\left(\frac{\operatorname{dim} M-2}{2}\right) .
$$
    
\end{prop}

\subsection{Computation of Euler characteristic}

The goal of this subsection is to prove Theorem \ref{orthogonal} for the regular elliptic numbers of $m$ of the root systems of type $B_n, D_n, {}^2D_n$ and all parities of $\dim M$ and $d$. 

\begin{thm}\label{orthogonal}
	We have 
	$$
-\chi_{c}(\widetilde{X}^{\circ}, \mathrm{Kl}^{\operatorname{st}}_{\widehat{G}, \mathbf{P}_{m}}(\chi , \phi)) = \begin{cases} 2d &  B_n, \\ 2d &  D_n, \text{ } \omega_1 \text{ is non-degenerate,} \\ 2(d+1) &  D_n, \text{ } \omega_1 \text{ is degenerate,} \\ 2d & ^{2}D_n,\text{ } \omega_1 \text{ is non-degenerate,} \\ 2(d+1) & ^{2}D_n,\text{ } \omega_1 \text{ is degenerate.}\end{cases}
$$

\end{thm} 
By a similar argument as in \cite{Katz} the Swan conductor of $ \mathrm{Kl}^{\operatorname{st}}_{\widehat{G}, \mathbf{P}_{m}}(\chi , \phi)$ at $\infty$ does not depend on $\chi$, so does the Euler characteristic of $\mathrm{Kl}^{\operatorname{st}}_{\widehat{G}, \mathbf{P}_{m}}(\chi , \phi)$. We make an inductive argument as follows.

\begin{prop}
Let $Q_i \subset \mathbb{P}(M)$ be the quadric defined by $q_{[\ell-i+1,m-\ell+i-1]}=0$. Let $U_i \subset \mathbb{P}(M_{0}\oplus_{j=\ell-i+1}^{m-\ell+i-1}M_j)$ be the projective variety defined by $Q(q)-\cup_{j=1}^{i} Q_j$ for $1\leq i \leq \ell$. Let $W_i \subset U_i$ be the quadric defined by $(\phi_{\ell-i}\cdots \phi_0 v_0, v_{m-\ell+i-1})=0$ for $1\leq i \leq \ell$. Assume $\chi =1$. We have
   \begin{align}\label{local system orthogonal}
   (-1)^{\mathrm{dim}M-2}\chi_c(\widetilde{X}^{\circ}, \mathrm{Kl}^{\operatorname{st}}_{\widehat{G}, \mathbf{P}_{m}}(1 , \phi)) &= -\chi_c(U_1)+\chi_c(W_1) 
\end{align}
\end{prop}

\begin{proof}
We assume $\mathcal{K}$ is the trivial character sheaf as in \cite[Prop 6.8]{YunEpipe}. We take the approach similar to the symplectic case.

Let $f_i: U_{\ell} \to \mathbb{A}^1$ be the function $[v] \mapsto \frac{(\phi_{i}v_{i}, v_{m-i-1})}{q_{[i+1,m-i-1]}}$ for $1 \leq i \leq \ell-1$. This function only depends on the coordinates $v_{i}, \cdots, v_{m-i}$. Let $f_{\geq \ell-i}=f_{\ell-i}+\cdots f_{\ell-1}$. 

Consider the projection $\pi_2: \mathbb{G}_m^{\mathrm{rot}} \times \mathfrak{G}_{\lambda} \to \mathfrak{G}_{\lambda}$. The stalk of $\pi_{2,!}f_{\phi}^{\ast}\mathrm{AS}_{\psi}$ over $[v]$ is 
\[
f^{\ast}_{\leq \ell}\mathrm{AS}_{\psi} \otimes \mathrm{H}^{\ast}_c(\mathbb{G}_m^{\mathrm{rot}}, T^{\ast}_{f_{0}([v])}\mathrm{AS}_{\psi}) 
\]
where $T_{f_{0}([v])}$ is the map $\mathbb{G}_m \to \mathbb{A}^1$ given by multiplication by $f_{0}([v])$. 

When $f_{0}([v])=0$ we have $\mathrm{H}^{\ast}_c (\mathbb{G}_m^{\mathrm{rot}}, T^{\ast}_{f_{0}([v])}\mathrm{AS}_{\psi})=\mathrm{H}^{\ast}_c (\mathbb{G}_m, 0^{\ast}\mathrm{AS}_{\psi})=\mathrm{H}^{\ast}_c(\mathbb{G}_m, \overline{\mathbb{Q}_{\ell}})$. When $f_{0}([v])\neq 0$, since $H_c^{\ast}(\mathbb{A}^1, \mathrm{AS}_{\psi})=0$, we have $\mathrm{H}^{\ast}_c (\mathbb{G}_m^{\mathrm{rot}}, T^{\ast}_{f_{0}([v])}\mathrm{AS}_{\psi})=\overline{\mathbb{Q}_{\ell}}[-1]$. Therefore we have
\begin{align}\begin{split}\label{chiind-orth}
(-1)^{\mathrm{dim}M-2}\chi_c(\widetilde{X}^{\circ}, \mathrm{Kl}^{\operatorname{st}}_{\widehat{G}, \mathbf{P}_{m}}(1 , \phi)) &= \chi_c(\mathfrak
    {G}_{\leq \lambda}, \pi_{2!}f_{\phi}^{\ast}\mathrm{AS}_{\psi}) \\
    &=  -\chi_c(U_{\ell}, f^{\ast}_{\geq 1}\mathrm{AS}_{\psi})+\chi_{c}(W_{\ell}, f^{\ast}_{\geq 1}\mathrm{AS}_{\psi})
    \end{split}
\end{align}

In the following we are going to show that for $1 \leq i \leq \ell-1$ we have
\begin{align}\label{Uind-orth}
    \chi_c(U_{i+1}, f^{\ast}_{\geq \ell-i} \mathrm{AS}_{\psi}) =  \chi_c(U_{i}, f^{\ast}_{\geq \ell-i+1} \mathrm{AS}_{\psi})
\end{align}
\begin{align}\label{Wind-orth}
    \chi_c(W_{i+1}, f^{\ast}_{\geq \ell-i} \mathrm{AS}_{\psi}) =  \chi_c(W_{i}, f^{\ast}_{\geq \ell-i+1} \mathrm{AS}_{\psi})
\end{align}

Let $U_i' \subset \mathbb{P}(M_{0}\oplus_{j=\ell-i}^{m-\ell+i-1}M_j)$ be defined by $Q(q)-\cup_{j=1}^{i} Q_j$. Consider the projection $p: U_{i+1} \to U_i'$ by forgetting the $M_{m-\ell+i}$ component. We have 
\[
p_{!}f_{\geq \ell-i}^{\ast}\mathrm{AS}_{\psi}  = f^{\ast}_{\geq \ell-i} \mathrm{AS}_{\psi}   \otimes p_{!}\overline{\mathbb{Q}_{\ell}}.
\]
Then we have
\[
\chi_c(U_{i+1}, f^{\ast}_{\geq \ell-i} \mathrm{AS}_{\psi}) = \chi_c(U_i', f^{\ast}_{\geq \ell-i} \mathrm{AS}_{\psi}\otimes p_{!}\overline{\mathbb{Q}_{\ell}}).
\]
Fix $[v']=[v_{0},v_{\ell-i},\cdots,v_{m-\ell+i-1}] \in U_i'$,  and we denote $q_{i} := q_{[\ell-i+1,m-\ell+i-1]}$ . The fiber of $p$ over $[v']$ is $\{ v_{m-\ell+i} | (v_{\ell-i}, v_{m-\ell+i})+q(v_{m-\ell+i})=0, (v_{\ell-i},v_{m-\ell+i})+q_i \neq 0 \}$. When $v_{\ell-i}\neq 0$, we have $H^{\ast}_c(p^{-1}([v'])) \cong H_c^{\ast}(\mathbb{G}_m)[-2d+4]$. When $v_{\ell-i}=0$, we have $H^{\ast}_c(p^{-1}([v'])) \cong \overline{\mathbb{Q}}_{\ell}[-2d-2]$. Since $U_i$ can be identified with the subscheme of $U_i'$ with $v_{\ell-i}=0$, we have 
\[
\chi_c(U_{i+1},  f^{\ast}_{\geq \ell-i} \mathrm{AS}_{\psi}) = \chi_c(U_i', f^{\ast}_{\geq \ell-i} \mathrm{AS}_{\psi}) = \chi_c(U_i, f^{\ast}_{\geq \ell-i+1} \mathrm{AS}_{\psi}).
\]

Consider the projection $p': W_{i+1} \to W_i$. We have
\[
p_{!}f_{\geq \ell-i}^{\ast}\mathrm{AS}_{\psi}  = f^{\ast}_{\geq \ell-i+1} \mathrm{AS}_{\psi}   \otimes p_{!}f_{\ell-i}^{\ast} \mathrm{AS}_{\psi}).
\]
We decompose $p$ into two steps $W_{i+1} \xrightarrow{p_1} W_i' \xrightarrow{p_2} W_i$, where $W_i' \subset \mathbb{P}(M_{0}\oplus_{j=\ell-i+1}^{m-\ell+i}M_j)$ is defined by $(\phi_{\ell-i-1}\cdots \phi_0v_0, v_{m-\ell+i})=0$. Fix $[v'']=[v_{0},v_{\ell-i+1},\cdots, v_{m-\ell+i}]$, and let $q_i:= q_i([v''])$. The fiber $p'^{-1}([v''])=\{ v_{\ell-i} | (v_{\ell-i}, v_{m-\ell+i})+q(v_{\ell-i})=0, (v_{\ell-i}, v_{m-\ell+i})+ q_i \neq 0) \}$. The function $f_{\ell-i-1}$ along the fiber is a linear function in $v_{\ell-i}$ by $f_{\ell-i}([v''])=\frac{(\phi_{\ell-i}v_{\ell-i}, v_{m-\ell+i})}{q_i}$. This the stalk of $p_{1,!}f^{\ast}_{\geq \ell-i} \mathrm{AS}_{\psi}$ at $[v'']$, which is $H^{\ast}_c(p^{-1}([v'']), f^{\ast}_{\ell-i}\mathrm{AS}_{\psi})$, vanishes when $v_{m-\ell+i} \neq 0$. Thus the stalk of $p_{2,!}p_{1,!}f_i^{\ast}\mathrm{AS}_{\psi}$ is isomorphic to $\overline{\mathbb{Q}}_{\ell}[-2d-2]$. Therefore, $p_{!}f_{\leq i+1}^{\ast}\mathrm{AS}_{\psi}$ and the constant sheaf $\overline{\mathbb{Q}_{\ell}}$ are the same in the Grothendieck group of $D_c^b(W_i)$. Thus we have 
\[
\chi_c(W_{i+1}, f_{\geq \ell-i}^{\ast}\mathrm{AS}_{\psi}) = \chi_c(W_i, f^{\ast}_{\geq \ell-i} \mathrm{AS}_{\psi}    p_{!}\overline{\mathbb{Q}_{\ell}} \mathrm{AS}_{\psi}) = \chi_c(W_i, f^{\ast}_{\geq \ell-i} \mathrm{AS}_{\psi}).
\]
Combining \cref{chiind-orth}, \cref{Uind-orth}, and \cref{Wind-orth} we get
\begin{align}
   (-1)^{\mathrm{dim}M-2}\chi_c(\widetilde{X}^{\circ}, \mathrm{Kl}^{\operatorname{st}}_{\widehat{G}, \mathbf{P}_{m}}(1 , \phi)) &= -\chi_c(U_1)+\chi_c(W_1) 
\end{align}
\end{proof}
\vspace{1em}
The proof of Theorem \ref{orthogonal} then reduces to the calculation of the compactly supported Euler characteristic of $U_1$ and $W_1$. We break these into a few lemmas listed below. Throughout the section,

\begin{itemize}
    \item $Q$ is the quadric on $\mathbb{P}(M_0 + M_\ell)$ defined by $q_0 + q_\ell=0$;
    \item $Q_1$ is the quadric on $\mathbb{P}(M_0 + M_\ell)$ defined by $q_0 =0$, $ q_\ell=0$ (this is because the whole space here is $Q(q)$ and $Q_1 \subset Q(q)$);
    \item $W_1$ is defiend by $\omega_{1} :=(\phi_\ell \cdots \phi_0(v_0), v_{\ell})=0 $ and $q_0+q_\ell=0$;
    \item The intersection $W_{1,1}:=W_1 \cap Q_1$ defined by  $\omega_{1,1}:=\omega_1 \cap Q_{1}$.
        
     \end{itemize}

\begin{lem}\label{Q_ell}
Suppose the root system is of type $B_n$. Then $$\chi(Q_1)=2d.$$
\end{lem}

\begin{proof}
	Without losing generality, we can assume $\dim M_0=d$, $\dim M_\ell =d+1$ where $d$ is even. We break $Q_1$ into two parts, depending on whether the projection of a point $(v_0,v_\ell)\in Q_1$ to $M_0$ is zero or not.
	
	When $v_0=0$, we have $v_\ell \ne 0.$ Call this part $Q_1^{v_0=0}$, it is just the quadric $q_\ell=0$ in $\mathbb{P}(M_{\ell}).$ Therefore the Euler characteristic of $Q_1^{v_0=0}$ is $d$. Here we do not have the primitive cohomology since the dimension of $M_{\ell}$ is odd, see \cite[Table 6.12]{YunEpipe}.
	
	When $v_0 \ne 0$, we call this part $Q_{1}^{v_0 \ne 0}$. Similar to what we did previously, we can project $Q_{1}^{v_0 \ne 0}$ onto the quadric $q_0=0$ in $\mathbb{P}(M_0)$ by sending $[(v_0,v_\ell)]$ to $[v_0]$. The fiber is a cone, so its Euler characteristic is $1$. The Euler characteristic of $Q_{1}^{v_0 \ne 0}$ is just the Euler characteristic of a quadric in $ P(M_0)$ which is $d$ since $d$ is even. The total Euler characteristic is the sum of the Euler characteristics of the two parts. \end{proof}
	
	\begin{lem}\label{three quadrics}
	Suppose the root system is of type $B_n$. Then $$\chi_c(W_{1,1}) =2d.$$
	\end{lem} 

\begin{proof}
The proof is similar to \ref{Q_ell}. We split $W_{1,1}$ into two (disjoint) parts
$$W_{1,1}=W_{1,1}^{v_0=0} + W_{1,1}^{v_0 \ne 0}$$
so that the Euler characteristic is the sum of the Euler characteristics of the two parts.

When $v_0=0$, we have $v_\ell \ne 0.$ It is the quadric $q_\ell=0$ in $\mathbb{P}(M_{\ell}).$ Therefore $W_{1,1}^{v_0=0}=d$.

When $v_0 \ne 0$, project $Q_1^{v_0 \ne 0}$ onto the quadric $q_0=0$ in $\mathbb{P}(M_0)$ by sending $[v_0,v_\ell]$ to $[v_0]$. The fiber is again a cone, so its Euler characteristic is $1$. The Euler characteristic of $Q_1^{v_0 \ne 0}$ is just the Euler characteristic of a quadric in $ P(M_0)$ which is $d$.
\end{proof}
\begin{remark}
	Computations of the Euler characteristic of $Q_1$ and $\omega_{1,1}$ of type $D_n$ and $^{2}D_n$ follow the same path. The reader should be aware that the different parity of $\dim M_0 \oplus M_\ell$, $\dim M_0$, and $\dim M_\ell$ lead to differences in the Euler characteristics, see \cite[Table 6.12]{YunEpipe}. We will omit the proof of Type $D_n$ and $^{2}D_n$ and write down $\chi_c(Q_1)$ and $\chi_c(W_{1,1})$ directly from the table by Yun and the previous two lemmas. We make a summarization in the following table. In the last but two and three columns, we list the Euler Characteristic of $W_{1}$ and $W_{1,1}$.

\end{remark}

\begin{table}[tb]
    \centering
    \label{table}
    \caption{Dimension, parity and Euler characteristics}
\begin{tabular}{ |p{0.8cm}||p{1.5cm}||p{1.5cm}||p{1.2cm}||p{1.1cm}||p{1.2cm}||p{1.1cm}||p{1.2cm}||p{1.3cm}||p{2cm}|}
  \hline
 Type & $\dim M_0\oplus M_\ell$ & Parity of $\dim M_0\oplus M_\ell$ &$\dim M_0$ & Parity of $\dim M_0$  & $\dim M_\ell$  & Parity of $\dim M_\ell$ & $\chi_c(W_{1})$ & $\chi_c(W_{1,1})$ & $-\chi_{c}(\mathrm{Kl}^{\operatorname{st}}_{\widehat{G}, \mathbf{P}_{m}})$ \\
 \hline
 \hline

 $B_n$   & 2d+1   & odd & d   & even & d+1 & odd & 4d & 2d& 2d\\
  &   2d+1 &  odd  & d+1 & even & d &odd & 4d & 2d & 2d\\
 $D_n$ & 2d & even & d  & even & d&even & 0 & 2d & 2d\\
  & 2d+2 & even & d+1& even &d+1 &even &0 &2(d+1) &2(d+1)\\
 $^{2}D_n$ &2d & even&d &odd   & d &odd &0 &2d-2 &2d\\
  &  2d+2 & even & d+1  &odd &d+1 &odd & 0&2d & 2(d+1)\\
 \hline    
\end{tabular}
\end{table}

\begin{proof}[Proof of Theorem \ref{orthogonal}]

We compute the Euler characteristic $-\chi_{c}(\widetilde{X}^{\circ}, \mathrm{Kl}^{\operatorname{st}}_{\widehat{G}, \mathbf{P}_{m}}(\chi , \phi))$ when the root system is of type $B_n$ with even $d=\dim M_0$, as an example to illustrate the usage of Table 1 in the proof of Theorem \ref{orthogonal}, and note that the computations of the Euler characteristic when the root system is of type $D_n$ or $^{2}D_n$ follows the same path.

When $\dim M_0 \oplus \dim M_ell =2d+1$, $Q$ is a smooth quadric and $W_{1}$ is either the intersection of two smooth quadrics or with singularities contains a factor isomorphic to $k^{*}$, both are in $\mathbb{P}(M_0 \oplus M_\ell)$. The table in \cite[Sec. 6]{YunEpipe} plus the fact that $\chi_c(k^{*})=0$ implies
$$\chi_c(Q)=2d, \text{  } \chi_c(W_1)=4d.$$

Therefore we have $$-\chi_{c}(\widetilde{X}^{\circ}, \mathrm{Kl}^{\operatorname{st}}_{\widehat{G}, \mathbf{P}_{m}}(\chi , \phi))=-\chi_c(Q)+\chi_c(Q_1)+\chi_c(W_1)-\chi_c(W_{1,1})=2d.$$

In the case of type $D_n$ and $^{2}D_n$, $\dim M$ (hence $\dim M -2$) is always even. So (\ref{local system orthogonal}) implies that 
$$-\chi_{c}(\widetilde{X}^{\circ}, \mathrm{Kl}^{\operatorname{st}}_{\widehat{G}, \mathbf{P}_{m}}(\chi , \phi))=\chi_c(U_1)-\chi_c(W_1)=\chi_c(Q)-\chi_c(Q_1)-\chi_c(W_1)+\chi_c(W_{1,1}).$$ Therefore one has to be careful with the signs of each term when calculating the Euler characteristic.
\end{proof}

\newpage
\bibliography{Kloo_symp.bib}
\bibliographystyle{alpha}
\end{document}